\documentclass[english]{article}
\usepackage[T1]{fontenc}
\usepackage[latin9]{inputenc}
\usepackage{babel}
\usepackage{amsmath}
\usepackage{esint}
\usepackage{amsfonts}
\usepackage{amssymb}
\usepackage{amsthm}
\usepackage{mathrsfs}
\usepackage{amscd}
\usepackage{caption}
\usepackage{subcaption}
\usepackage{fullpage}
\usepackage{graphicx}
\DeclareGraphicsExtensions{.pdf,.png,.jpg}
\usepackage{setspace}
\graphicspath{./figure/}
\date{}

\newcommand{\Q}{\mathbb{Q}}
\newcommand{\Z}{\mathbb{Z}}

\newtheorem{prop}{Proposition}[subsection]

\newtheorem{lemma}{Lemma}[subsection]

\newtheorem*{theorem*}{Theorem}

\theoremstyle{remark}
\newtheorem{remark}{Remark}[subsection]

\theoremstyle{definition}
\newtheorem{defin}{Definition}[subsection]

\begin{document}

\title{Index polynomials for virtual tangles}

\author{Nicolas Petit}

\maketitle
\begin{center}Oxford College of Emory University (petitnicola@gmail.com)\end{center}

\begin{abstract}
\noindent We generalize the index polynomial invariant, originally introduced in \cite{cobordismofknotsonsurfaces}, \cite{henrich} and \cite{longframedfti}, to the case of virtual tangles. Three polynomial invariants result from this generalization; we give a brief overview of their definition and some basic properties.\end{abstract}

\section{Introduction}

Ever since its introduction in \cite{henrich} and \cite{cobordismofknotsonsurfaces}, index polynomial invariants of virtual knots have been a very popular topic, see for example \cite{affineindexpolynomial}, \cite{linkingnumberaffineindex}, \cite{indexpolyinvariantoflinks}.
Because of their construction, many of these turn out to be Vassiliev invariants of order one for virtual knots, a notion first introduced in \cite{virtualknottheory} as a natural generalization of finite-type invariants of classical knots.
The author himself worked on generalizing the results of \cite{henrich} to the cases of framed and/or long virtual knots \cite{longframedfti}, resulting in (amongst others) a polynomial invariant for long virtual knots which happens to be a Vassiliev invariant of order one. 

It then seemed natural to wish to extend said results, and in particular the Henrich-Turaev polynomial, to the case of virtual tangles. 
We define three different polynomial invariants of virtual tangles, including two infinite families parametrized by rational numbers $a,b$, all of which are finite-type invariants of order one for virtual tangles that generalize the Henrich-Turaev polynomial.
We show that these invariants for a sequence of increasing strength, and that if we limit ourselves to the case of virtual string links (possibly with permutation of the endpoints) the three polynomials are all additive under connected sum.

This paper is organized as follows: section \ref{virtualknotsandtangles} provides a brief review of the definition of virtual knots, virtual tangles, and linking number, that will be needed in the rest of the paper. Section \ref{longknotinvariants} reviews Vassiliev invariant and briefly summarizes the results of \cite{henrich} and \cite{longframedfti}. We then define the polynomial invariants for tangles in section \ref{tangleinvariants}, before studying some of their properties in section \ref{strengthofinvariants}.

\section{Background}
\subsection{Virtual knots and tangles}
\label{virtualknotsandtangles}
Virtual knots were first introduced by Kauffman in \cite{virtualknottheory}. They can be defined in multiple ways: we can think of a virtual knot as a knot projection with three types of crossing (classical positive, classical negative and virtual), up to Reidemeister moves and virtual Reidemeister moves:

\begin{figure}[!h]
\centering
\includegraphics[scale=0.1]{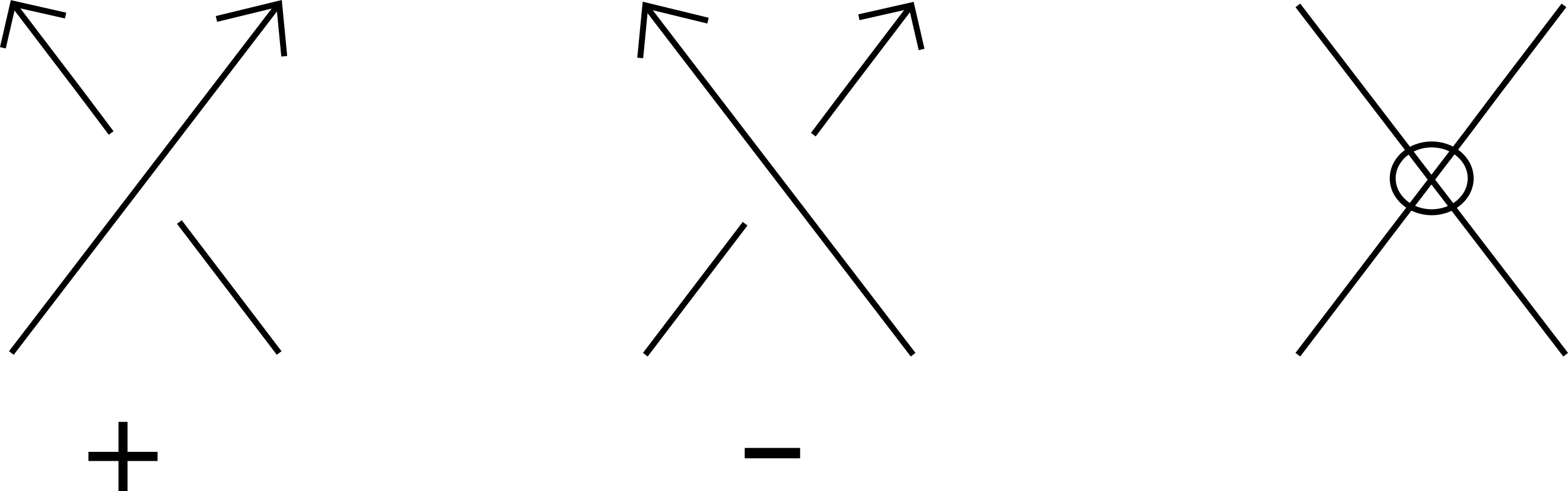}
\caption{The three types of crossing: positive, negative and virtual.}
\end{figure}

\begin{figure}[!h]
\centering
\includegraphics[scale=.08]{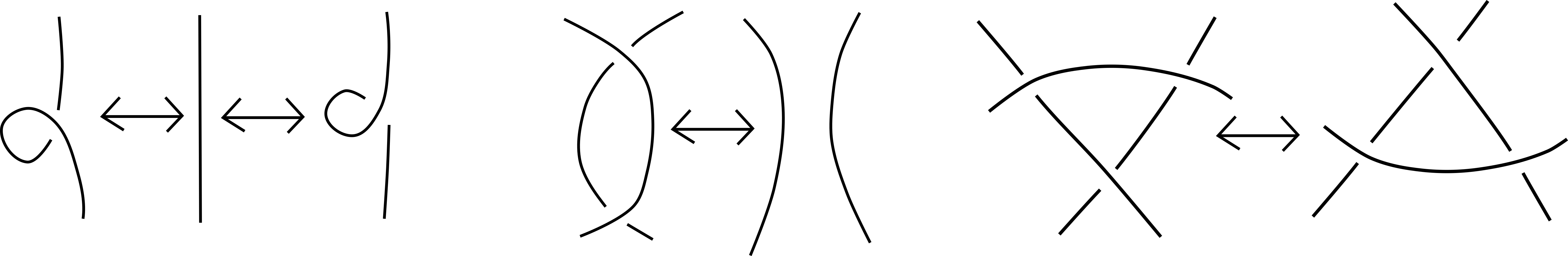}
\caption{The classical Reidemeister moves.}
\label{RM}
\end{figure}

\begin{figure}[!h]
\centering
\includegraphics[scale=.08]{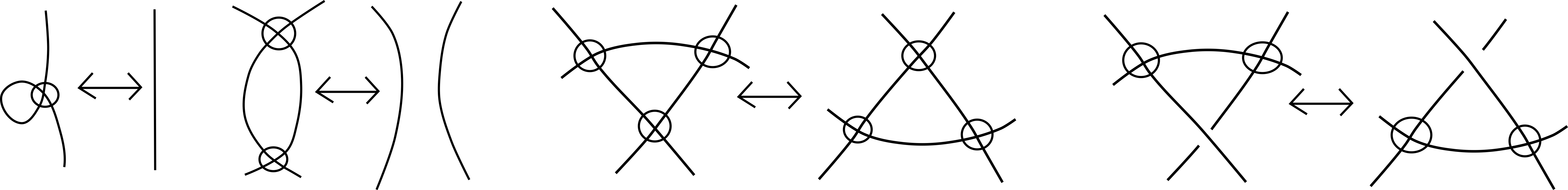}
\caption{The virtual Reidemeister moves.}
\label{virtualRM}
\end{figure}

We can also think of virtual knots as equivalence classes of Gauss diagrams up to Reidemeister moves, or as equivalence classes of knots in thickened surfaces up to stabilization/destabilization. We will mostly focus on the diagrammatic approach in this paper.

We will also work with flat knotlike objects in this paper: given a diagrammatic knot theory in terms of Reidemeister moves, the associated flat theory is obtained by adding the crossing change (CC) move, pictured in Fig. \ref{CCmove}. This move allows to let a strand cross another, which for virtual knots means we're working under homotopy rather than isotopy. 
Since the over/under information of the crossings is now irrelevant, we typically draw these knots as immersed curves, with the implicit understanding that either one of the two resolutions can be picked.

\begin{figure}[!h]
\centering
\includegraphics[scale=.12]{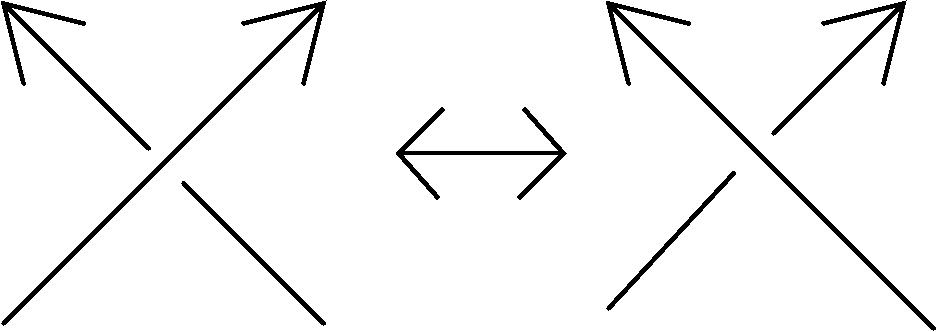}
\caption{The CC move.}
\label{CCmove}
\end{figure}

A long virtual knot is a long virtual knot diagram with classical and virtual crossings up to Reidemeister moves. 
They can also be represented by Gauss diagrams with a distinguished point (representing the ``point at infinity'') or as embeddings of knots in thickened surfaces with boundary. 
For more details about how these definitions are related, see \cite{longframedfti}.

Finally, a virtual tangle is a collection of virtual knots and long virtual knots, typically linked with each other in some fashion. 
We will typically represent tangles in a square, where the ends of the long component are fixed distinguished points on the top or bottom side of the square. 
If the tangle has $m$ endpoints at the top and $n$ endpoints at the bottom, we say that the tangle is an $(m,n)$ tangle. 
We usually draw the top and bottom of the square, but not the sides, see Fig. \ref{virtualtangleexample}.

We can define the connected sum of two virtual tangles as the ``stacking'' of one tangle above another. 
Of course, this operation is only defined if the two tangles have the same number of distinguished points alongside the common boundary, and their orientations match; because of this, the operation is, generally speaking, not commutative. 
We will denote the connected sum of tangle $T$ and $U$ by $T\#U$, in which we stack $T$ above $U$.

\begin{remark}
When talking about a ``tangle'' we will typically work in the most general setting: we allow our tangle to contain closed components, and we don't assume that it has the same number of distinguished points in both boundaries, or that long components necessarily connect top to bottom.
\end{remark}

\begin{figure}[!h]
\centering
\includegraphics[scale=.13]{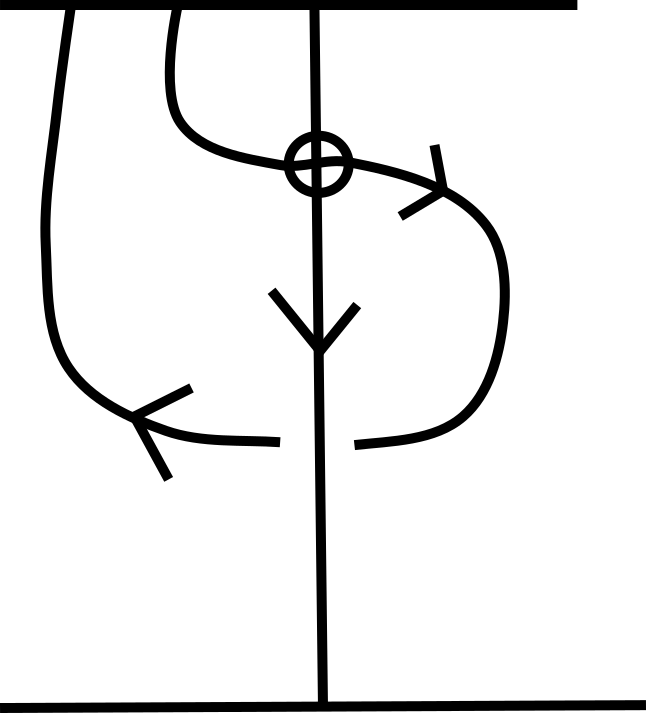}
\caption{An example of a $(3,1)$ virtual tangle.}
\label{virtualtangleexample}
\end{figure}

A special case we'll consider is that of an $(n,n)$ virtual tangle where each long component connects a distinguished point at the top with a distinguished point at the bottom.
These tangles could possibly have closed components, and we can always take the connected sum of two such tangles if they're unoriented.
These tangles lead to two different specializations: braids on $n$ stands and virtual string links.

A virtual braid on $n$ strands is an $(n,n)$ tangle, with no closed component, where each component connects a distinguished point at the top with a distinguished point at the bottom, and in which every strand is always moving downwards (so the strands have no local extrema).
Virtual braids on $n$ strands form a group, called the virtual braid group, and every virtual link can be represented as a virtual braid on a given number of strands \cite{virtualbraidsLmove}.
There is a rich literature on virtual braids, which we encourage the interested reader to consult.
 
A virtual sting link on $n$ components is a special type of $(n,n)$ tangle, with no closed component, with the same set of distinguished points at the top and bottom of the square, and such that the components connect the same distinguished point at the top and bottom. 
This means that there is a natural labeling on the components (given by the order of the points). We also make the choice to always orient virtual string links from top to bottom.
With all these observations, it is clear that virtual string links with $n$ components form a monoid, with the identity element given by the pure braid on $n$ strands.
Because the ordering and orientation of the components is well-defined, we can represent virtual string links as Gauss diagrams with $n$ core circles (typically drawn as straight lines, since all the components are long).  An example of a $3$-component virtual string link is in Fig. \ref{virtualstringlinkexample}.
Also note that the closure of an $n$-component virtual string link is an $n$-component virtual link.

\begin{figure}[!h]
\centering
\includegraphics[scale=.16]{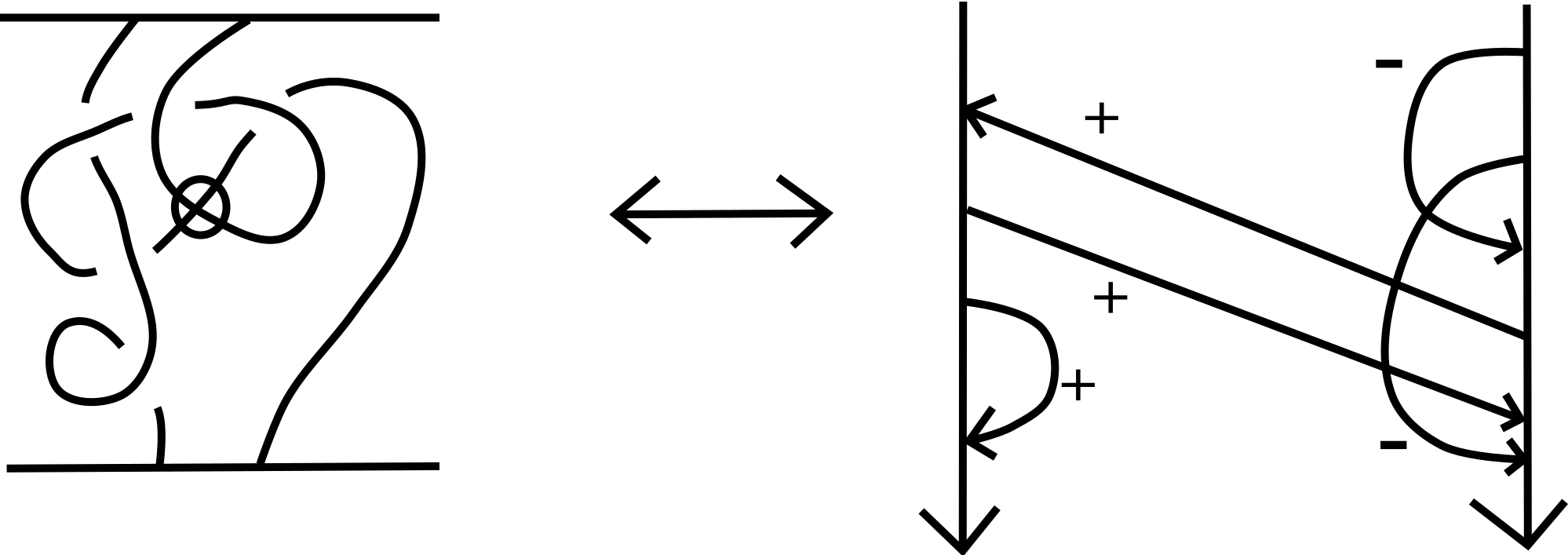}
\caption{A simple virtual string link on two components, and its Gauss diagram.}
\label{virtualstringlinkexample}
\end{figure}

Finally, let's briefly review the notion of linking number of a virtual link.
For oriented classical links, we have three equivalent definitions of the linking number between two components: we can either take the writhe of all the crossings where component one goes over component two, or take the writhe of all crossings where component two goes over component one, or take the overall writhe of the two components and divide by two.
$$lk(T_1, T_2)=\frac{1}{2}\sum_{d\in T_1\cap T_2}sgn(d)=\sum_{d\in T_1\text{ over } T_2}sgn(d)=\sum_{d\in T_2\text{ over } T_1}sgn(d)$$
As a result of the definition, we clearly have that $lk(T_1, T_2)=lk(T_2, T_1)$, so the linking number only depends on the pair of components.

However, that fact is not true anymore for virtual links, as the number of times $T_1$ goes over $T_2$ need not be the same as the number of times $T_2$ goes over $T_1$. To see this, just take any classical link and change one crossing from classical to virtual.
To be able to talk about a virtual linking number we must then pick one of the definitions: we will use the over definition, so $$vlk(T_1, T_2)=\sum_{d\in T_1\text{ over } T_2}sgn(d).$$
Note that the above observations imply that $vlk(T_1, T_2)\neq vlk(T_2, T_1)$, so we need to consider the \emph{ordered} pair of components $(T_1, T_2)$ when computing a virtual linking number. The difference $vlk(T_1, T_2)-vlk(T_2, T_1)$ is called the wriggle number $W(T_1, T_2)$ of the pair of components \cite{linkingnumberaffineindex}; clearly the wriggle number is an invariant of the ordered pair, and $W(T_2, T_1)=-W(T_1, T_2)$, so $|W(T_1, T_2)|$ is an invariant of the (unordered) pair of components.

\begin{lemma}
\label{anylinkingnumber}
Given any pair of linking numbers $vlk(L_1, L_2), vlk(L_2, L_1)$, there is a virtual link $L$ that realizes those linking numbers.
\end{lemma}

\begin{proof}
We can easily produce a Gauss diagram of such a link. The one in Fig. \ref{gaussanylinkingnumber} has $vlk(L_1, L_2)=a$ and $vlk(L_1, L_2)=-b$.
\end{proof}

\begin{figure}[!h]
\centering
\includegraphics[scale=.12]{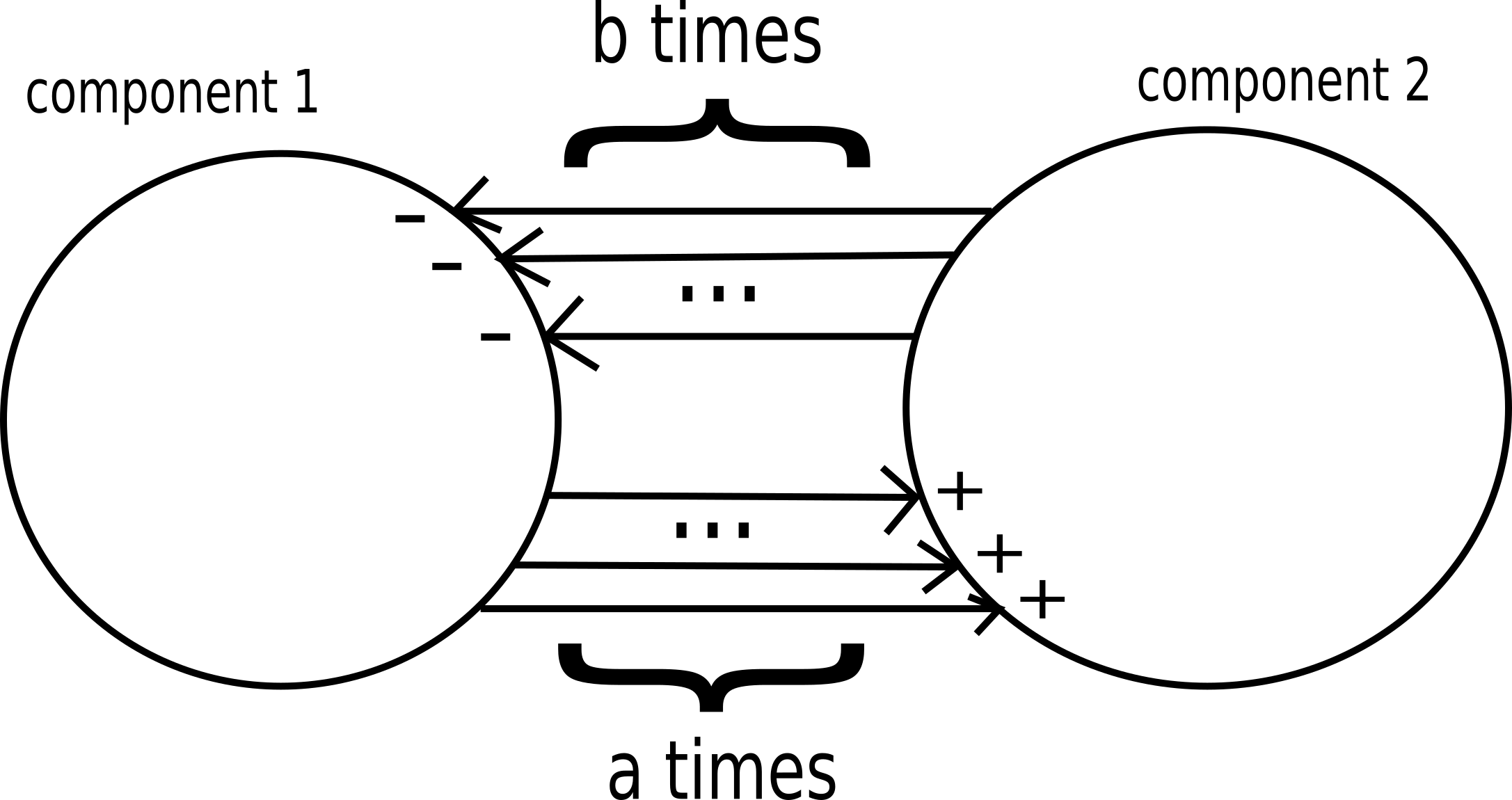}
\caption{The Gauss diagram of a link with $vlk(L_1, L_2)=a$ and $vlk(L_2, L_1)=-b$.}
\label{gaussanylinkingnumber}
\end{figure}

\begin{remark}
A diagram that realizes such a link can be obtained by starting with the unlink, applying the virtualization move to a crossing to switch the over/under stands, then adding extra crossings via a Reidemeister move 2 and changing one of the crossings to be virtual, see Fig. \ref{anylinkingnumbers}.
\end{remark}

\begin{figure}[!h]
\centering
\includegraphics[scale=.095]{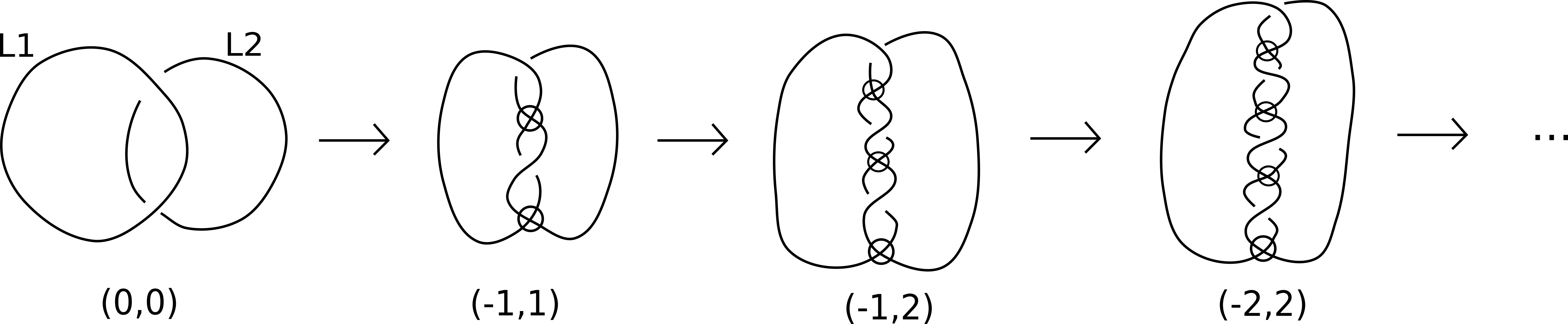}
\caption{How to construct a virtual link diagram with any possible pair of linking numbers. The ordered pairs under the pictures are the values $(vlk(L_1, L_2), vlk(L_2, L_1))$.}
\label{anylinkingnumbers}
\end{figure}

%---------------------------------------------------------------------------------

\subsection{Vassiliev invariants of virtual knots}
\label{longknotinvariants}
Vassiliev invariants for virtual knots were first introduced by Kauffman in \cite{virtualknottheory} as a natural generalization of the notion of a finite-type invariant for classical knots.
Formally, a Vassiliev invariant of virtual knots is the extension of a virtual knot invariant to the category of singular virtual knots.
These are virtual knots with an extra type of crossing (transverse double points), modulo the Reidemeister moves for virtual knots and some extra moves, called rigid vertex isotopy.

We extend the virtual knot invariant $\nu$ to singular virtual knots by resolving every double point as a weighted average of its two possible resolution, a positive and a negative crossing, see Fig. \ref{doublepointresolution}.
We then say that $\nu$ is a finite-type invariant, or Vassiliev invariant, of order $\leq n$ if it vanishes on every knot with more than $n$ double points (i.e. the $n+1$st derivative vanishes). 
Examples of finite-type invariants according to this definition are the coefficient of $x^n$ in the Conway polynomial or the Birman coefficients. 
\begin{figure}[!h]
\centering
\includegraphics[scale=.07]{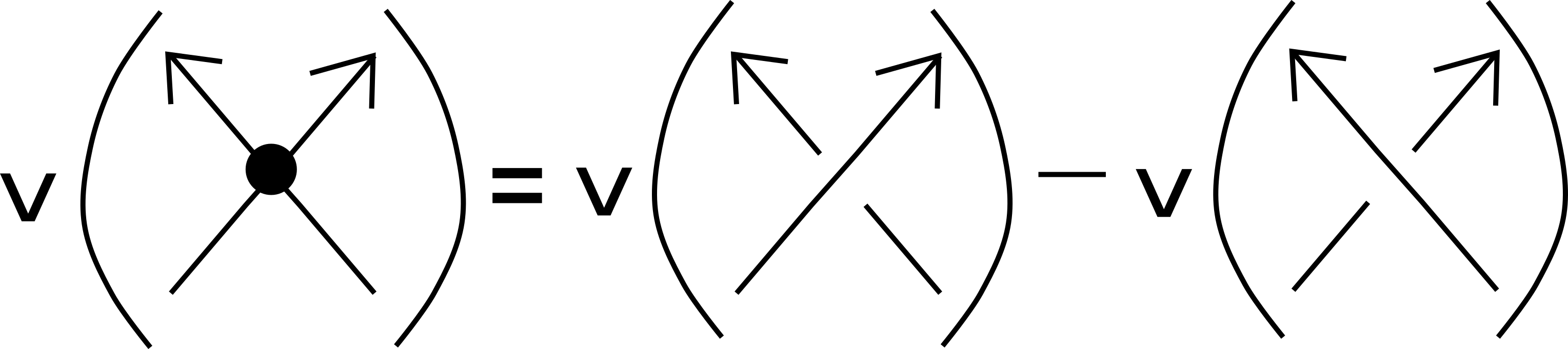}
\caption{How to resolve a double point.}
\label{doublepointresolution}
\end{figure}

We can naturally extend the above ideas to the case of virtual tangles. A finite-type or Vassiliev invariant of virtual tangles will be an extension of a virtual tangle invariant to virtual tangles with double points, where we resolve a double point with the same relation pictured in Fig. \ref{doublepointresolution}. 
Then $\nu$ is a Vassiliev invariant of order $\leq n$ for virtual tangles if it vanishes on any tangle with more than $n$ double points; we say $\nu$ is of order $n$ if it is of order $\leq n$ but not of order $\leq n-1$.
As far as we know, this is the first use of Vassiliev invariants of virtual tangles in the literature.

\begin{remark}
Around the same time, Goussarov, Polyak and Viro \cite{GPV} were independently developing a different generalization to virtual knots of finite-type invariants, in terms of a new type of crossing called semi-virtual. We will not work with these finite-type invariants in this paper, though the author and a collaborator are looking into the relation between the invariants of section \ref{tangleinvariants} and GPV invariants.
\end{remark}

Little work seems to have been done in the theory of Vassiliev invariants of virtual knots until \cite{cobordismofknotsonsurfaces}, \cite{henrich}, whose results were then generalized by the author to the case of long and/or framed virtual knots. 
While \cite{henrich} has a sequence of three Vassiliev invariants of increasing strength, we will only focus on generalizing the polynomial invariant (also called the Henrich-Turaev polynomial) to the case of virtual tangles; as such, we recall here those definitions and theorems.

\begin{remark}
We typically smooth crossings according to orientation and give the resulting object the inherited orientation, see Fig. \ref{smoothingofacrossing}.
If both stands belong to the same component, as a result of the smoothing we will get a two-component object.
\end{remark}

\begin{figure}[!h]
\centering
\includegraphics[scale=.1]{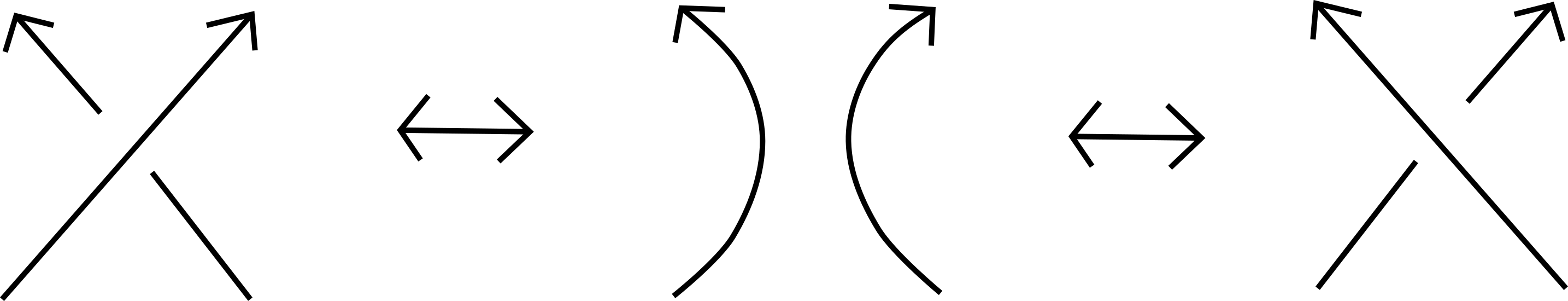}
\caption{How to smooth a crossing according to orientation. Note that positive and negative crossings give the same oriented smoothing.}
\label{smoothingofacrossing}
\end{figure}

\begin{defin}
Take a (long) virtual knot and a classical crossing $d$ in it, and smooth the crossing according to orientation. The result of the operation is a $2$-component link: pick an arbitrary ordering of the two components.
Now consider the flat version of this ordered virtual link, and assign a sign to each crossing involving both components as in Fig. \ref{flatlinksign} (mnemonic: component $1$ lies above component $2$).

\begin{figure}[!h]
\centering
\includegraphics[scale=.1]{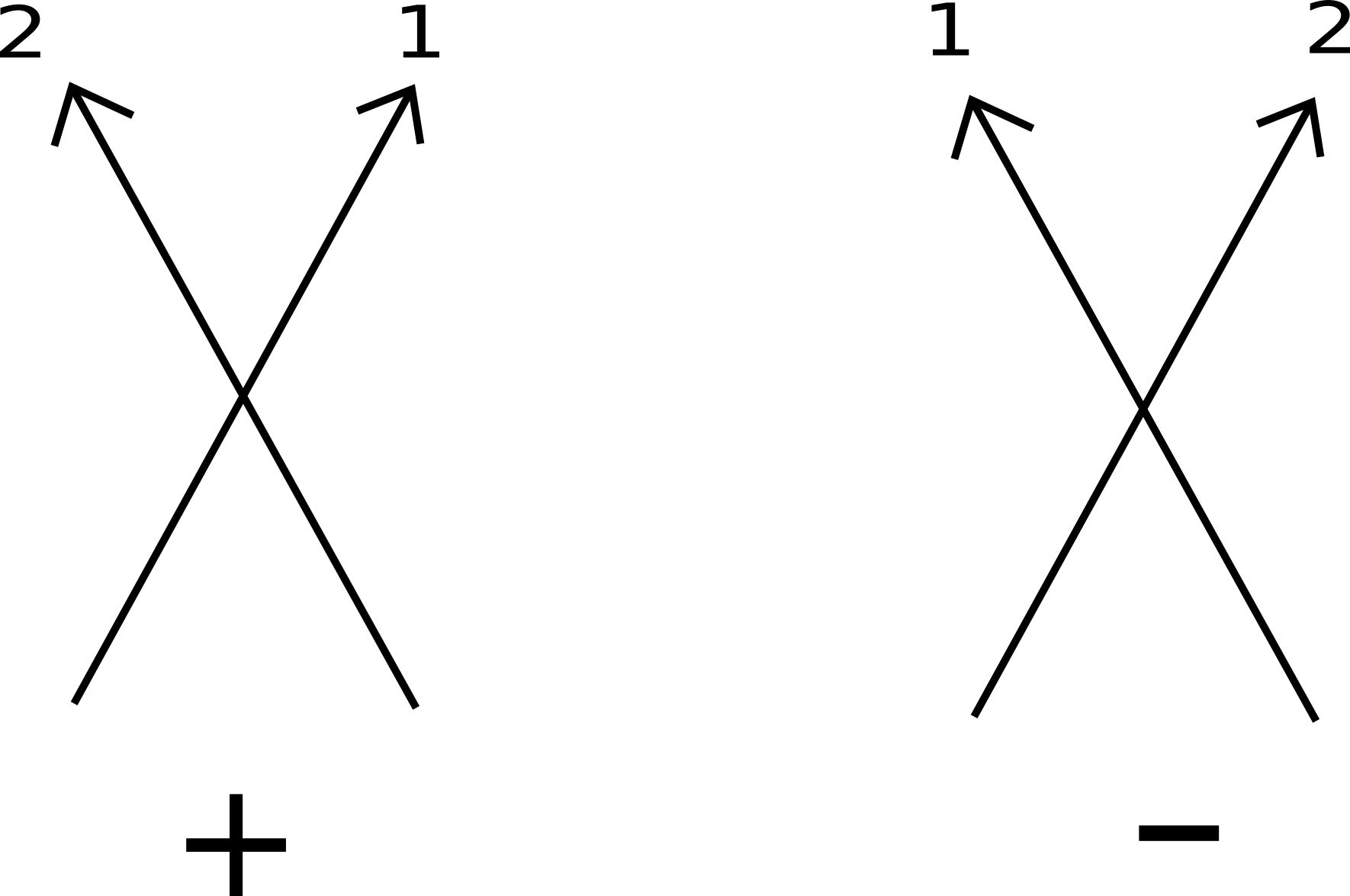}
\caption{How to compute the sign of a crossing involving both components.}
\label{flatlinksign}
\end{figure}

The intersection index of $d$ is the quantity defined by $$i(d)=\sum_{c\in L_1\cap L_2} sgn(c)$$
where $c$ is a crossing involving both components.
\end{defin}

\begin{remark}
Note that a different choice of component $1$ and $2$ would change all the signs, so $i(d)$ isn't well-defined, but $|i(d)|$ is.
\end{remark}

\begin{prop}[\cite{henrich}]
Let $K$ be a virtual knot. The polynomial $$p_t(K)=\sum_{c}sgn(c)(t^{|i(d)|}-1)\in \Z[t],$$where $c$ is taken over all the classical crossings of $K$, is an order one Vassiliev invariant of virtual knots.
\end{prop}

\begin{prop}[\cite{longframedfti}]
If $K$ is a long virtual knot, the polynomial $$p_t(K)=\sum_{c}sgn(c)(t^{|i(d)|}-1)\in \Z[t],$$ where $c$ is taken over all the classical crossings of $K$, is an order one Vassiliev invariant of long virtual knots.
Suppose that in the computation of $i(d)$ you systematically order the components by always picking the long component to be component one. Then $i(d)$ is now well-defined, and $$\overline{p_t}(K)=\sum_{c}sgn(c)(t^{i(d)}-1)\in \Z[t]$$ is an order one Vassiliev invariant of long virtual knots, which is stronger than $p_t(K)$.
\end{prop}

%------------------------------------------------------------------------------------

\section{The virtual tangle invariants and their properties}

\subsection{The polynomial invariants}
\label{tangleinvariants}
Because it is our goal to generalize the Henrich-Turaev polynomial to virtual tangles, we will use it as a starting point (see section \ref{longknotinvariants}). 
It turns out that we need to consider polynomials in multiple variables to be able to generalize the polynomial invariant.

\begin{defin}
Let $T$ be an unoriented virtual tangle on $n$ ordered components $T_1,\ldots, T_n$, and assign to each component a variable $t_i$.
The self-crossing polynomial is the polynomial in $\Z[t_1,\ldots, t_n]$ defined by
$$p_{sc}(T)=\sum_{i=1}^n\sum_{c\in T_i} sgn(c)(t_i^{|i(c)|}-1)$$
where $c$ is a classical self-crossing of component $T_i$, $sgn(c)=\pm 1$ is the sign of the crossing, and $i(c)$ is the intersection index defined in section \ref{longknotinvariants}. 
In computing the intersection index, we ignore the crossings involving other components of the tangle besides the two we created in the smoothing.
\end{defin}

\begin{remark}
Changing the orientation of a strand doesn't have an effect on the sign of a self-crossing (as is the case for knots), so the quantity $sgn(c)$ is well-defined for an unoriented tangle.
\end{remark}

\begin{prop}
The self-crossing polynomial is a tangle FTI of order $\leq 1$.
\end{prop}

\begin{proof}
To check whether the polynomial is a virtual tangle invariant, we must ensure invariance under the tangle isotopy moves. Since only self-crossings are involved in the definition, we only need to verify the cases where there is at least a classical crossing and where both strands in a crossing belong to the same component.
In the first Reidemeister move case, after smoothing the crossing we get two components that don't intersect, so the total contribution of a kink is $\pm (t_i^0-1)=0$, which is the same as the contribution of the straight strand with no kinks.
In the second Reidemeister move case, the two crossings have opposite signs, and the flat tangles obtained after smoothing the crossings are homotopic (hence they have the same intersection index), so the two contributions cancel each other out.
In every other case, there is a correspondence between the crossings on either side of the move that respects both the signs of the crossings and their intersection indices, so the contributions on both sides of a move are the same.
In fact, the correspondence respects not just the intersection index, but also the flat class of the tangles obtained after the smoothing (from which we recover $i(c)$).

Now that we proved that $p_{sc}$ is a virtual tangle invariant, we need to check it's a FTI of order $\leq 1$.
The argument why it vanishes on any tangle with two double points is the same as in the virtual knot case: suppose that $d, d'$ are the two double points of tangle $T$.
We then expand the double points as \begin{equation}T_{dd'}=T_{++}-T_{+-}-T_{-+}+T_{--},\label{eqresolution}\end{equation}
where $+$ and $-$ represent the sign we picked in the resolution of $d, d'$.
If $d$ or $d'$ is a double point involving distinct components, either resolution will not contribute to the value $p_{sc}$. 
If the double point is a self-crossing, we note that the intersection index is the same for the positive and negative resolutions; because each resolution appears in both a positive and negative term in the expansion of equation \ref{eqresolution}, the overall contribution of each double point is zero.
Finally, the contributions coming from all the other self-crossings are the same in each term of the expansion; because of the alternating signs in the expansion, the overall contribution of these terms is also zero.
This shows that $p_{sc}(T_{dd'})=0$ no matter what $d,d'$ are.

We can also easily provide an example of a singular virtual tangle with $p_{sc}(T)\neq 0$ by repurposing the long virtual knot example from \cite{longframedfti}; after all, long virtual knots are just tangles with one component.
Figure \ref{pscnonzero} shows a tangle for which $p_{sc}(T)\neq 0$: resolving the double point one way gives (after some work) a string with no crossing, while the other resolution has a nonzero value.

\begin{figure}[!h]
\centering
\includegraphics[scale=.15]{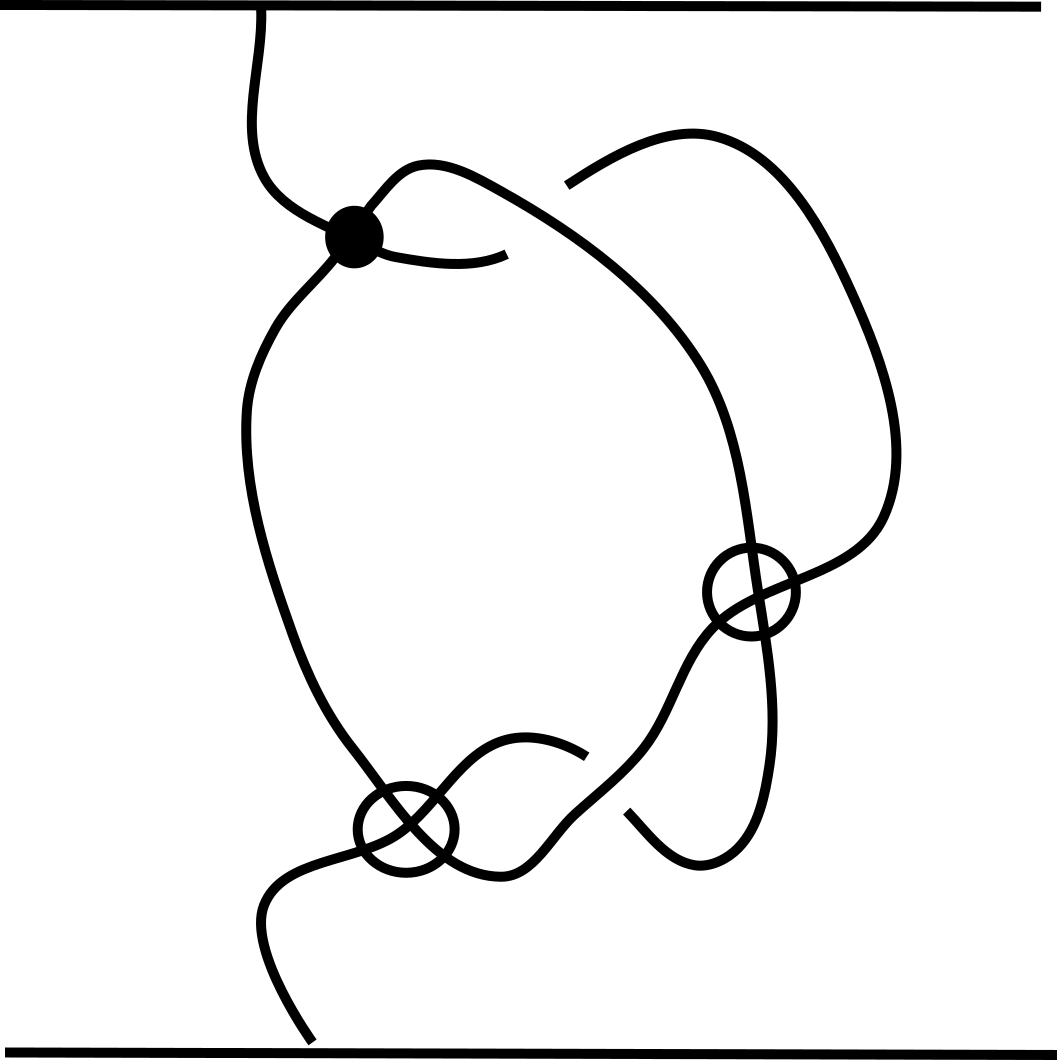}
\caption{A singular virtual tangle (in fact, a singular long virtual knot) for which $p_{sc}$ is nonzero.}
\label{pscnonzero}
\end{figure}

\end{proof}

While the self-crossing polynomial is a generalization of the Henrich-Turaev polynomial to the virtual tangle case, it does not see the way that the components are linked, which makes it a very coarse invariant.
We aim to solve this issue by extending the self-crossing polynomial to two infinite families of polynomial invariants, which we call the linking number polynomials.

\begin{defin}
Let $T$ be an oriented virtual tangle on $n$ components, whose components $T_1,\ldots, T_n$ are ordered, and assign to each component a variable $t_i$. Fix two rational numbers $a,b\in \Q$.
The linking number polynomials are the polynomials in the variables $t_1,\ldots, t_n$ defined by
$$p_{lk}(T)=p_{sc}(T)+\sum_{\substack{i\neq j\\ i<j}}[a\cdot vlk(T_i, T_j)+b\cdot vlk(T_j, T_i)]t_it_j$$
$$p_{lk,L}(T)=p_{sc}(T)+\sum_{\substack{i\neq j\\ i<j}}a\cdot vlk(T_i, T_j)t_it_j^{-1}+b\cdot vlk(T_j, T_i)t^{-1}_it_j$$
where $vlk$ is the virtual linking number defined in section \ref{virtualknotsandtangles}.
\end{defin}

\begin{remark}
A few remarks about these definitions:
\begin{itemize}
\item The invariant $p_{lk, L}$ is technically a Laurent polynomial, not a polynomial (hence the subscript $L$).
\item We include the $p_{sc}$ terms because we still require these polynomials to generalize the Henrich-Turaev polynomial.
\item We must restrict our attention to oriented virtual tangles because the sign of a crossing involving two components will change if we swap the orientation of only one of the components.
\item In this paper, we will assume that the numbers $a,b$ are rational and fixed at the beginning of the computation, and that they are the same for each pair of components. The author is interested in analyzing how the invariants behave when different pairs of components get different coefficients, or what the relation is between $p_{lk}$ and $p_{lk, L}$ with different values of $a$ and $b$, in future work. 
\item In the definition of $p_{lk}$, the choice $a=b$ makes the coefficient of $t_it_j$ a multiple of the total writhe between the components $T_i, T_j$, while the choice $a=-b$ makes the coefficient a multiple of the Wriggle number $W(T_i, T_j)$.
\end{itemize}
\end{remark}

\begin{prop}
The linking number polynomials are order one Vassiliev invariants of virtual tangles.
\end{prop}

\begin{proof}
Since the self-crossing polynomial is an invariant and the extra terms we added do not influence the value of $p_{sc}$, to check that the linking number polynomials are invariants we just need to verify the invariance under the Reidemeister moves of the new terms. 
Also note that Reidemeister move one only adds/removes self-crossings, so we can restrict our attention the the other Reidemeister moves.
However, the same arguments that worked for the self-crossing polynomial work in this case: the second Reidemeister move adds two crossings of opposite signs involving the same pair of strands, so the overall contribution is zero, while Reidemeister move three has a correspondence between the crossings (and their signs) before and after the move, so the contribution is the same on both sides of the move.

To show that our polynomials vanish on any tangle with two double points, a similar argument to the self-crossing polynomial case holds, albeit we need to be careful about what strands are glued at the double point.
If a double point involves two strands of the same component, the intersection index is the same for both the positive and the negative resolution, and each of these appears both with a positive and a negative sign in front of it, so the overall contribution of that double point is zero.
If a double point involve two strands belonging to different components $T_i$ and $T_j$, we need to keep track of the sign of the resolved crossing, as well as which strand lies on top. 
But the resolution where strand $T_i$ is on top appears twice in the expansion, once with positive coefficient and once with negative coefficient, so its contribution is zero, and similarly the resolution where strand $T_j$ is on top appears twice, so its total contribution is also zero.
Finally, every other crossing will appear twice with the sign $+$ and twice with the sign $-$, so their overall contribution is zero. 
This concludes the argument that $p_{lk}(T)$ and $p_{lk, L}(T)$ are always zero if $T$ has two double points.

Finally, Fig. \ref{pscnonzero} is still an example of a tangle with nonzero oriented polynomial (as there is only one component whose $p_{sc}$ is nonzero).
\end{proof}

\begin{remark}
These polynomials could be slightly generalized using the ``ordered polynomial'' definition from \cite{longframedfti} for the self-crossing polynomial.
\end{remark}

%------------------------------------------------------------------
\subsection{Strength and properties of the polynomials}
\label{strengthofinvariants}

As a reminder, invariant $A$ is stronger than invariant $B$ if whenever $A(K)=A(K')$ we also get $B(K)=B(K')$, and there is a pair of tangles for which $B(T_1)=B(T_2)$ but $A(T_1)\neq A(T_2)$.

\begin{prop}
The linking polynomial invariants are stronger than the self-crossing polynomial.
The Laurent polynomial $p_{lk, L}$ is stronger than $p_{lk}$.
\end{prop}

\begin{proof}
Since the linking polynomials contain $p_{sc}$ and the extra terms don't influence the coefficients of $t_i^k$ for all $k$, if two tangles have the same linking number polynomial they also clearly have the same self-crossing polynomial.
If $a\neq -b$, two tangles with $p_{sc}=0$ but different values of the linking polynomials are shown in Fig. \ref{orderedstrongerthansc}: the tangle on the left is unlinked, so all the polynomials are zero, while the two crossings in the tangle on the right are both positive, so the linking polynomials are $$p_{lk}(T_2)=(a+b)t_1t_2\qquad p_{lk, L}(T_2)=at_1t_2^{-1}+bt_1^{-1}t_2$$

\begin{figure}[!h]
\centering
\includegraphics[scale=.15]{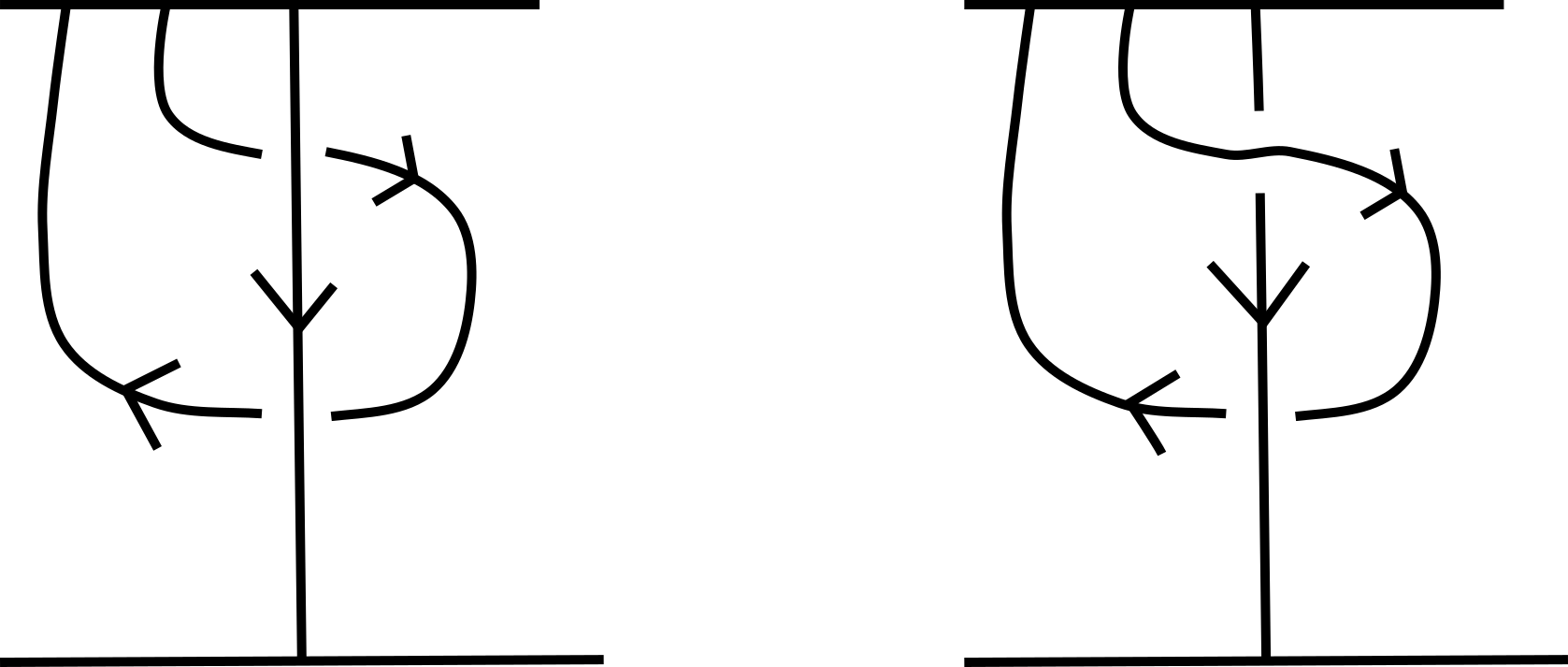}
\caption{Two tangles with the same self-crossing polynomial $p_{sc}=0$, but different linking polynomials (if $a\neq -b$). If $a=-b$ we can find a suitable pair by turning the bottom crossing virtual.}
\label{orderedstrongerthansc}
\end{figure}

If $a=-b$ the tangles in Fig. \ref{orderedstrongerthansc} still have different values of $p_{lk, L}$, but they now have the same value of $p_{lk} =0$ (so we can't use them to show $p_{lk}$ is stronger than $p_{sc}$). 
We can fix this by changing the bottom crossing to be virtual instead of classical; then one of the tangles will have $p_{lk}(T_1)=at_1t_2$, while the other has $p_{lk}(T_2)=bt_1t_2$, and since $a=-b$ we indeed get $p_{lk}(T_1)\neq p_{lk}(T_2)$.

If two tangles have the same $p_{lk, L}$, they also must have the same $p_{lk}$: from the expression of $p_{lk, L}$ we know that both tangles have the same values of $p_{sc}$ as well as the same values of linking numbers for every possible pair of components (as previously observed, the two linking numbers don't affect each other in the expression of $p_{lk, L}$).
As a result, the values $a\cdot vlk(T_i, T_j)+b\cdot vlk(T_j, T_i)$ is also the same for both tangles for every choice of $i,j$, so the two tangles have the same $p_{lk}$.

Finally, you can easily produce a non-trivial tangle that has $p_{lk}=0$: since $a,b\in\Q$, you can find two integers $m,n$ for which $am+bn=0$ (in fact, there are infinitely many such pairs). 
Then pick the two-strand braid $B$ whose closure is the two-component link with linking numbers $vlk(T_1, T_2)=m$, and $vlk(T_2, T_1)=n$, which we know exists by the Lemma \ref{anylinkingnumber}. 
Note that this braid has no self-crossings, so by construction $p_{lk}(B)=0$; this means that $p_{lk}$ is unable to distinguish this braid from the pure braid on two strands.
However, $p_{lk, L}(B)=amt_1t_2+bnt_2t_1\neq 0$, so the Laurent version successfully distinguishes $B$ from the pure braid on two strands, and is hence stronger than $p_{lk}$.
\end{proof}

\begin{prop}
If $T, U$ are virtual string links on $n$ components, all three polynomials are additive: $p(T\# U)=(p(T)+p(U))$. Moreover, since the polynomials take value in a commutative ring, $p(T\# U)=p(U\# T)$.

If $T, U$ are $(n,n)$ tangles where each long component connects a distinguished point at the top with one at the bottom, we can compute the connected sum $T\#U$, and component $T_i$ glues to component $U_j$, then all three polynomials are additive modulo identifying which component glues to which.

\end{prop}

\begin{proof}
Since virtual string links are always oriented from top to bottom, we can always take the connected sums $T\#U$ and $U\#T$; moreover, since each component goes from the $i$-th endpoint at the top to the $i$-th endpoint at the bottom, when we stack the tangles together we get another virtual string link whose $i$-th component is just the result of glueing the $i$-th component of $T$ with the $i$-th component of $U$.
As a result of that, the set of crossings between the $i$-th and $j$-th component of $T\#U$ is the union of the crossings between the $i$-th and $j$-th component of $T$ and the $i$-th and $j$-th component of $U$, and the same is true of $U\#T$.
Finally, note that the two components we get after smoothing a crossing in $T$ will not intersect $U$, so the connected sum has no effect on the index of a given crossing if, like in this case, the connected sum doesn't join two components of $T$. 
Of course, a similar observation is valid for the indices of the classical crossings of $U$.
The above observations imply that $p(T\#U)=p(T)+p(U)=p(U\#T)$, where $p$ is any of the three polynomial invariants.

Suppose we're dealing with the slightly more general case of two $(n,n)$ tangles $T, U$ where each long component connects an endpoint at the top to an endpoint at the bottom; these are essentially virtual string links under a permutation of the distinguished points, and possibly also having some closed components.
Each tangle has its own self-crossing polynomial $p_{sc}(T)\in\Z[t_1, \ldots, t_n]$, $p_{sc}(U)\in\Z[u_1, \ldots, u_n]$ (we use different variables to keep track of the components more easily).
Assuming that the connected sum $T\#U$ is well-defined (there could be issues with orientation), each long component of $T$ must be connected to a long component of $U$, say $T_i$ is glued to $U_j$ to form the $i$-th component of $T\#U$; let $R$ be the set of $n$ relations of the form $t_i=u_j$ that come from the above identification.
On the one hand the $i$-th component of $T\#U$ contributes to the self-crossing polynomial the terms
$$\sum_{c\in (T\#U)_i}sgn(c)(t_i^{|i(c)|}-1).$$

But the self-crossings of $(T\#U)_i$ are the union of the self-crossings of $T_i$ with those of $U_j$, and (as above) the values of the index of each crossing is unaffected by the connected sum. 
We can then rewrite the above expression as

$$\sum_{c\in (T\#U)_i}sgn(c)(t_i^{|i(c)|}-1)=\sum_{\substack{c\in (T\#U)_i\\c\in T_i}}sgn(c)(t_i^{|i(c)|}-1)+\sum_{\substack{c\in (T\#U)_i \\ c\in U_j}}sgn(c)(t_i^{|i(c)|}-1).$$
On the other hand, the sum of the contributions of the self-crossings of $T_i$ and the self-crossings of $U_j$ is
$$\sum_{c\in T_i}sgn(c)(t_i^{|i(c)|}-1)+\sum_{d\in U_j}sgn(d)(u_j^{|i(d)|}-1).$$
So up to identifying the dummy variables $t_i$ and $u_j$ (which we do by quotienting out by $R$), the terms are the same as the contribution of $(T\#U)_i$; if we took the sum over all components, we would thus get
$$p_{sc}(T\#U)=p_{sc}(T)+p_{sc}(U)\in \Z[t_1, \ldots, t_n, u_1, \ldots, u_n]/R.$$

The same reasoning extends to the linking number polynomials: suppose that component $T_i$ attaches to component $U_l$ to form $(T\#U)_i$, and component $T_j$ attaches to component $U_m$ to form $(T\#U)_j$, which gives us the same set $R$ of relations as in the self-crossing polynomial case.
The crossings between components $(T\#U)_i$ and $(T\#U)_j$ correspond to the union of the crossings between $T_i$ and $T_j$ and the crossings between $U_l$ and $U_m$; where $T_i$ goes over $T_j$, or $U_l$ goes over $U_m$, we have that $(T\#U)_i$ goes over $(T\#U)_j$.
Thus $$vlk((T\#U)_i, (T\#U)_j)=vlk(T_i, T_j)+vlk(U_l, U_m)\qquad vlk((T\#U)_j, (T\#U)_i)=vlk(T_j, T_i)+vlk(U_m, U_l).$$
Using these identities, and quotienting by the set $R$ of relations, we then get that
$$p_{lk}(T\#U)=p_{lk}(T)+p_{lk}(U)\in \Z[t_1, \ldots, t_n, u_1, \ldots, u_n]/R$$$$ p_{lk,L}(T\#U)=p_{lk,L}(T)+p_{lk,L}(U)\in \Z[t_1^{\pm1}, \ldots, t_n^{\pm1}, u_1^{\pm1}, \ldots, u_n^{\pm1}]/R$$

Note that for this type of tangles $T\#U$ and $U\#T$ are not necessarily both defined (the orientations could clash in one of the connected sums but not the other), so generally speaking $p(T\#U)\neq p(U\#T)$.

\end{proof}

\begin{remark}
For the most general tangle case, we are skeptical that any of these polynomial invariants is additive; the main issue is that connected sum with an appropriate tangle (containing, for example, a cup or a cap) could merge two components that are separate in the original tangle, which could influence the value of the intersection index.
\end{remark}

\section{Future work}
We consider this the first in what will (hopefully) be a series of papers on the topics of virtual tangles, index polynomials and finite-type invariants. We are currently investigating a series of different projects related to the results of this paper, of which we list the ones of most immediate interest. 
The lack of additivity of the polynomials for general tangles (which we hope to prove soon) is our greatest regret. 
Since every knot can be decomposed as the connected sum of virtual tangles, we would ultimately like to have a polynomial invariant for tangles that is additive and that returns the Henrich-Turaev polynomial (or some other index polynomial) when applied to the tangle decomposition of a virtual knot.
We are currently working on relating the invariants in this paper with the affine index polynomial and the wriggle polynomial from \cite{linkingnumberaffineindex}, in hopes of finding a definition that would satisfy the above requirements.

The author and a collaborator are also working on reinterpreting these results in terms of GPV finite-type invariants. A thorough study of GPV invariants for tangles seems to be missing from the literature, so we hope to fill that gap.
As a long-term project, we would like to generalize the other two invariants (smoothing and glueing) from \cite{henrich} to the case of virtual tangles, and possibly find a universal finite-type invariant of order one for virtual tangles.

The author would like to acknowledge his institution, Oxford College of Emory University, for research support during the development of this paper.

\bibliographystyle{amsalpha}
\bibliography{fullbibliography}

\providecommand{\bysame}{\leavevmode\hbox to3em{\hrulefill}\thinspace}
\providecommand{\MR}{\relax\ifhmode\unskip\space\fi MR }
\providecommand{\MRhref}[2]{%
  \href{http://www.ams.org/mathscinet-getitem?mr=#1}{#2}
}
\providecommand{\href}[2]{#2}
\begin{thebibliography}{GPV00}

\bibitem[FK13]{linkingnumberaffineindex}
L.~C. Folwaczny and L.~H. Kauffman, \emph{A linking number definition of the
  affine index polynomial}, Journal of Knot Theory and Its Ramifications
  \textbf{22} (2013), no.~12, 1341004.

\bibitem[GPV00]{GPV}
M.~Goussarov, M.~Polyak, and O.~Viro, \emph{Finite-type invariants of classical
  and virtual knots.}, Topology \textbf{39} (2000), no.~5, 1045--1068.

\bibitem[Hen10]{henrich}
A.~Henrich, \emph{A sequence of degree one vassiliev invariants for virtual
  knots.}, Journal of Knot Theory and its ramifications \textbf{19} (2010),
  no.~4, 461--487.

\bibitem[ILL10]{indexpolyinvariantoflinks}
Y.~H. Im, K.~Lee, and S.~Y. Lee, \emph{Index polynomial invariant of virtual
  links}, Journal of Knot Theory and Its Ramifications \textbf{19} (2010),
  no.~05, 709--725.

\bibitem[Kau99]{virtualknottheory}
L.~H. Kauffman, \emph{Virtual knot theory}, European J. Comb \textbf{20}
  (1999), no.~7, 663--690.

\bibitem[Kau13]{affineindexpolynomial}
\bysame, \emph{An affine index polynomial invariant of virtual knots}, Journal
  of Knot Theory and Its Ramifications \textbf{22} (2013), no.~04, 1340007.

\bibitem[KL06]{virtualbraidsLmove}
L.~H. Kauffman and S~Lambropoulou, \emph{Virtual braids and the l-move},
  Journal of Knot Theory and Its Ramifications \textbf{15} (2006), no.~06,
  773--811.

\bibitem[Pet]{longframedfti}
N.~Petit, \emph{Finite-type invariants of long and framed virtual knots},
  https://arxiv.org/abs/1610.03825; submitted for publication.

\bibitem[Tur08]{cobordismofknotsonsurfaces}
V.~Turaev, \emph{Cobordism of knots on surfaces}, Journal of Topology
  \textbf{1} (2008), no.~2, 285--305.

\end{thebibliography}


\begin{thebibliography}{1}
\expandafter\ifx\csname url\endcsname\relax
  \def\url#1{\texttt{#1}}\fi
\expandafter\ifx\csname urlprefix\endcsname\relax\def\urlprefix{URL }\fi
\providecommand{\bibinfo}[2]{#2}
\providecommand{\eprint}[2][]{\url{#2}}


\bibitem{GPV}
\bibinfo{author}{M. Goussarov, M. Polyak, and O. Viro.}
\newblock \bibinfo{title}{{Finite-type invariants of classical and virtual knots.}}
\newblock \emph{\bibinfo{journal}{Topology}}
  \textbf{\bibinfo{volume}{39(5)}}, \bibinfo{pages}{1045--1068}
  (\bibinfo{year}{2000}).
  
  \bibitem{henrich}
\bibinfo{author}{A. Henrich}.
\newblock \bibinfo{title}{{A sequence of degree one Vassiliev invariants for virtual knots.}}
\newblock \emph{\bibinfo{journal}{Journal of Knot Theory and its ramifications}}
  \textbf{\bibinfo{volume}{19(4)}}, \bibinfo{pages}{461--487}
  (\bibinfo{year}{2010}).

\bibitem{virtualknottheory}
\bibinfo{author}{Louis H. Kauffman}.
\newblock \bibinfo{title}{{Virtual Knot Theory}}.
\newblock \emph{\bibinfo{journal}{European J. Comb}}
  \textbf{\bibinfo{volume}{20(7)}}, \bibinfo{pages}{663--690}
  (\bibinfo{year}{1999}).
  
  \bibitem{stateoftheart}
\bibinfo{author}{V. O. Manturov and D. Ilyutko}.
\newblock \bibinfo{title}{{Virtual Knots: State of the Art}}.
\newblock \emph{\bibinfo{publisher}{World Scientific}}
  (\bibinfo{year}{2012}).

  \bibitem{virtualstrings}
\bibinfo{author}{V. Turaev}.
\newblock \bibinfo{title}{{Virtual Strings.}}
\newblock \emph{\bibinfo{journal}{Annales de l'Institut Fourier}}
  \textbf{\bibinfo{volume}{54(7)}}, \bibinfo{pages}{2455--2525}
  (\bibinfo{year}{2004}).
  
  

\end{thebibliography}

\end{document}